\documentclass[11pt]{amsart}

\usepackage{amsfonts}
\usepackage{amssymb}
\usepackage{setspace}
\usepackage{comment}

\usepackage{geometry}
\usepackage[utf8]{inputenc}
\usepackage[english]{babel}
\usepackage{enumitem}
 
\usepackage{amsthm}
\usepackage{amsmath}
\usepackage{amssymb}

\newtheorem{theorem}{Theorem}

\newtheorem{lemma}{Lemma}

\theoremstyle{remark}

\newtheorem*{rmk}{Remark}

\renewcommand{\leq}{\leqslant}
\renewcommand{\geq}{\geqslant}

\title{A note on exponential-M\"{o}bius sums over $\mathbb{F}_q[t]$ }

\author[S. Porritt]{Sam Porritt}
\address{Department of Mathematics\\University College London\\
25 Gordon Street, London, England}
\email{samuel.porritt.15@ucl.ac.uk}

\begin{document}


\maketitle

\begin{abstract}
In 1991, Baker and Harman proved, under the assumption of the generalized Riemann hypothesis, that $\max_{ \theta \in [0,1) }\left|\sum_{ n \leq x } \mu(n) e(n \theta) \right| \ll_\epsilon x^{3/4 + \epsilon}$. The purpose of this note is to deduce an analogous bound in the context of polynomials over a finite field using Weil's Riemann Hypothesis for curves over a finite field. Our approach is based on the work of Hayes who studied exponential sums over irreducible polynomials.
\end{abstract}

\section{Introduction}

Let $\mu$ be the M\"{o}bius function and write $e(\theta) = e^{2\pi i \theta}$. Baker and Harman~\cite{BaHar} proved under the assumption of the generalized Riemann hypothesis that for all $\epsilon > 0$,
\begin{equation}
\max_{\theta \in [0,1)} \left| \sum_{n \leq x} \mu(n) e(n\theta) \right| \ll_\epsilon x^{\frac{3}{4}+\epsilon}.
\end{equation}
It is conjecture that (1) holds for all $\epsilon>0$ with $\frac{3}{4}$ replaced by $\frac{1}{2}$. The best unconditional result is due to Davenport~\cite{Dav} who showed that for all $A > 0$
$$
\max_{\theta \in [0,1)} \left| \sum_{n \leq x} \mu(n) e(n\theta) \right| \ll_A \frac{x}{(\log x)^A}.
$$

The purpose of this note is to deduce an analogue of (1) for the polynomial ring $\mathbb{F}_q[t]$. First, let us go through some definitions required to state the result. The function field analogue of the real numbers is the completion of the field of fractions of $\mathbb{F}_q[t]$ with respect to the norm defined by
$$|f/g| = \begin{cases} q^{\deg f - \deg g} &\text{ if } f\neq 0 \\0 & \text{ otherwise}. \end{cases}$$
This completion is naturally identified with the ring of formal Laurent series $\mathbb{F}_q((1/t))=\{ \sum_{i\leq j} x_i t^i \: : \: x_i \in \mathbb{F}_q, \: j \in \mathbb{Z} \}.$ The norm defined above is extended to $x = \sum_{ i \leq j } x_i t^i \in \mathbb{F}_q((1/t))$ by setting $|x| = q^j$ where $j$ is the largest index with $x_j \neq 0$. The analogue of the unit interval is $\mathbb{T}:=\{ \sum_{i < 0} x_i t^i \: : \: x_i \in \mathbb{F}_q \},$ and is a subring of $\mathbb{F}_q((1/t))$.

Define the additive character $\psi : \mathbb{F}_q \rightarrow \mathbb{C}^\times$ by $\psi(x) = e(\text{tr}(x)/p)$, where $\text{tr} : \mathbb{F}_q \rightarrow \mathbb{F}_p$ is the usual trace map and $p$ is the characteristic of $\mathbb{F}_q$. Define also the exponential map $\textbf{e}_q : \mathbb{F}_q((1/t))\rightarrow \mathbb{C}^\times$ by $\textbf{e}_q (x) = \psi(x_{-1}).$

Now let $\mu$ denote the M\"{o}bius function on the ring $\mathbb{F}_q[t]$. All sums over polynomials are sums over monic polynomials.

\begin{theorem}\label{thm}
Suppose $ n \geq 3$. Then
$$\max_{\theta \in \mathbb{T}} \left| \sum_{ \deg f = n }\mu(f) \textbf{e}_q(f \theta) \right| \leq 4 q^{\frac{3n+1}{4}}\left(\tfrac{3\sqrt{3}}{2}\right)^n.$$
\end{theorem}

\begin{rmk}
It follows that for all $ \epsilon > 0$ and $q$ large enough with respect to $\epsilon$ we have
$$\max_{\theta \in \mathbb{T}} \left| \sum_{ \deg f = n }\mu(f) \textbf{e}_q(f \theta) \right| \leq 4 q^{(\frac{3}{4} + \epsilon )n}.$$
\end{rmk}

Our proof of Theorem~\ref{thm} will follow the strategy of Hayes employed in his study of the exponential sum
$$\sum_{ \substack{ \deg \omega = n \\ \omega \text{  irreducible  }  }} \textbf{e}_q(\omega \theta).$$
Recently, Bienvenu and L\^e have independently derived a similar result to Theorem~\ref{thm} in~\cite{PiHo}. Their Theorem 9 corresponds to our Lemma~\ref{mobchar} and their Theorem~11 closely resembles our Theorem~\ref{thm}.

\emph{Acknowledgements}. We are very grateful to Pierre Bienvenu for pointing out a mistake in an earlier version of our proof of Theorem~\ref{thm}. We are also grateful to Andrew Granville for pointing out how to strengthen an earlier version of Theorem~\ref{thm}. Previously, we required both $n$ and $q$ to be large with respect to $\epsilon$ for the remark that follows Theorem~\ref{thm} to hold. This work was supported by the Engineering and Physical Sciences Research Council EP/L015234/1 via the
EPSRC Centre for Doctoral Training in Geometry and Number Theory (The London School of Geometry and
Number Theory), University College London.

\section{Lemmas}

Let $\mathbb{F}_q[t]^\times$ be the multiplicative monoid of monic polynomials in $\mathbb{F}_q[t]$. Whilst investigating the distribution of irreducible polynomoials over $\mathbb{F}_q$, Hayes~\cite{Hay1} introduced certain congruences classes on $\mathbb{F}_q[t]^\times$ defined as follows. Let $s \geq 0$ be an integer and $g \in \mathbb{F}_q[t]$. We define an equivalence relation $\mathcal{R}_{s,g}$ on $\mathbb{F}_q[t]^\times$  by
$$ a \equiv b \text{ mod } \mathcal{R}_{s,g} \Leftrightarrow g \text{ divides } a-b \text{ and }  \left|\frac{a}{t^{\deg a}} - \frac{b}{t^{\deg b}}\right|<\frac{1}{q^{s}}$$
It is easy to check that this is indeed an equivalence relation and that for all $c \in \mathbb{F}_q[t]^\times$,
$$a \equiv b \text{ mod } \mathcal{R}_{s,g} \Rightarrow ac \equiv bc \text{ mod } \mathcal{R}_{s,g}$$
so we can define the quotient monoid $ \mathbb{F}_q[t]^\times \slash \mathcal{R}_{s,g}$. Hayes showed that an element of $\mathbb{F}_q[t]$ is invertible modulo $\mathcal{R}_{s,g}$ if and only if it is coprime to $g$ and that the units of this quotient monoid form an abelian group of order $q^s \phi(g)$ which we denote $\mathcal{R}_{s,g}^* = \left(\mathbb{F}_q[t]^\times \slash \mathcal{R}_{s,g} \right)^\times.$ Given a character (group homomorphism) $\chi : \mathcal{R}_{s,g}^* \rightarrow \mathbb{C}$ we can lift this to a character of $\mathbb{F}_q[t]^\times$ by setting $\chi(f)=0$ if $f$ is not invertible modulo $\mathcal{R}_{s,g}$. Associated to each such character is the $L$-function $L(u, \chi)$ defined for $u \in \mathbb{C}$ with $|u|<1/q$ by
$$L(u, \chi) = \sum_{ f \in \mathbb{F}_q[t]^\times } \chi(f) u^{\deg f } = \prod_{ \omega }(1-\chi(\omega)u^{\deg \omega})^{-1}$$
where the product is over all monic irreducibles. When $\chi$ is a non-trivial character it can be shown that $L(u, \chi)$ is a polynomial which factorises as
$$L(u,\chi) =  \prod_{i = 1}^{d(\chi)} (1-\alpha_i(\chi)u)$$ for some $d(\chi) \leq s + \deg g -1 $ and each  $\alpha_i(\chi)$ satisfies $|\alpha_i(\chi)| = 1$ or $\sqrt{q}$. This follows from Weil's Riemann Hypothesis and appears to have been first proved by Rhin in~\cite{Rhin}. 

When $\chi = \chi_0$ is the trivial character we have
$$L(u , \chi_0) =  \sum_{\substack{ f \in \mathbb{F}_q[t]^\times \\(f,g)=1 }}u^{\deg f} =\sum_{ f \in \mathbb{F}_q[t]^\times }u^{\deg f}\prod_{\omega | g }(1-u^{\deg \omega}) = \frac{1}{1-qu}\prod_{\omega | g }(1-u^{\deg \omega}).$$

\begin{lemma}\label{mobchar}
Let $\chi$ be a character modulo $\mathcal{R}^*_{s,g}$ and $\deg g \leq n/2$. Then
$$\left|\sum_{ \deg f = n } \mu(f)\chi(f) \right | \leq
\begin{cases}
{{n+s+\deg g -2}\choose{s+ \deg g -2}} q^{n/2} \:\:\: &\text{  if  } \chi \neq \chi_0 \\
\binom{n+r-1}{r-1}(q+1) &\text{ if } \chi = \chi_0
\end{cases}
$$
where $ r $ is the number of distinct irreducible divisors of $g$.
\end{lemma}

\begin{rmk}
The bound $\chi_0$ is smaller than the one for $\chi \neq \chi_0$ when $n \geq 3$ because $\deg g$ is an upper bound for $r$ and for $n \geq 3$
$$(q+1)\binom{n+ \deg g - 1}{n} \leq \binom{n+ \deg g - 2}{n} q^{n/2}.$$
\end{rmk}

\begin{proof}
Suppose first that $\chi \neq \chi_0$. Then
\begin{align*}
\sum_{ f } \chi(f) \mu(f) u^{\deg f} = L(u,\chi)^{-1} = \prod_{i = 1}^{d(\chi)}(1- \alpha_i(\chi)u)^{-1}
= \sum_{n \geq 0} \left( \sum_{ \substack{r_1 + \cdots r_{d(\chi)} = n \\ 0 \leq r_i \leq n    }} \prod_{i = 1}^{d(\chi)} \alpha_i(\chi)^{r_i} \right)u^n.
\end{align*}
Comparing coefficients and using the triangle inequality we get
\begin{align*}
\left| \sum_{ \deg f = n } \chi(f) \mu(f) \right| = \left| \sum_{ \substack{ r_1 + \cdots + r_{d(\chi)} = n \\ 0 \leq r_i \leq n    } } \prod_{i = 1}^{d(\chi)} \alpha_i(\chi)^{r_i} \right| &\leq {{n + d(\chi) - 1}\choose{d(\chi) -1 }} q^{n/2} \\
&\leq {{n+s+\deg g -2}\choose{s+ \deg g -2}} q^{n/2}.
\end{align*}
When $\chi= \chi_0$ is the principal character
$$L(u,\chi_0)^{-1} = (1-qu)\prod_{\omega|g}(1+u^{\deg \omega} + u^{2 \deg \omega} + \cdots ).$$
If we write $\omega_1, \omega_2, \ldots, \omega_r$ for the distinct irreducible divisors of $g$ then we get, by equating coefficients again,
\begin{align*}
\left|\sum_{\deg f = n}\chi_0(f)\mu(f) \right| &\leq \sum_{ \substack{ a_i \in \mathbb{Z}_{\geq 0}\\ \sum_{1\leq i \leq r}a_i \deg \omega_i = n }}1 + q\sum_{ \substack{ a_i \in \mathbb{Z}_{\geq 0}\\ \sum_{1\leq i \leq r}a_i \deg \omega_i = n-1 }}1 \\
&\leq (q+1)\sum_{ \substack{ b_i \in \mathbb{Z}_{\geq 0}\\ \sum_{1\leq i \leq r}b_i = n }}1 \\
&= (q+1)\binom{n+r-1}{r-1}.
\end{align*}
\end{proof}

\begin{lemma}\label{decomp}
For each $\theta \in \mathbb{T}$ there exist unique coprime polynomials $a, g \in \mathbb{F}_q[t]$ with $g$ monic and $\deg a < \deg g \leq n/2$ such that
$$ \left| \theta - \frac{a}{g} \right| < \frac{1}{q^{\frac{n}{2} + \deg g}}.$$
\end{lemma}

\begin{proof}
See Lemma 3 from~\cite{Pol}.
\end{proof}

\begin{lemma}\label{lem}
Let $\theta \in \mathbb{T}$ and let $a, g$ be the unique polynomials defined as in Lemma 2 with respect to $\theta$ and $n$. Set $ s = n - [\frac{n}{2}] - \deg g$. For any $f_1,f_2 \in \mathbb{F}_q[t]^\times$ of degree $n$ such that $f_1 \equiv f_2 \text{ mod } \mathcal{R}_{s,g}$ we have
$$\textbf{e}_q(f_1 \theta) = \textbf{e}_q(f_2 \theta) .$$
\end{lemma}

\begin{proof}
See Lemma 5.2 from~\cite{Hay2}.
\end{proof}

\begin{lemma}\label{better}
Suppose $g \in \mathbb{F}_q[t]$ is square-free. Then
$$\sum_{d | g} \frac{1}{q^{\deg d}} \leq (1+\frac{\log (\deg g)}{\log q})e.$$
\end{lemma}

\begin{proof}
Order the monic irreducibles $\omega_1, \omega_2, \ldots, \omega_r$ dividing $g$ and the monic irreducibles $P_1, \ldots$ in $\mathbb{F}_q[t]$ in order of degree (and those of the same degree arbitrarily). Let $\pi(k)$ be the number of monic irreducibles of degree $k$ and define $N$ by $\sum_{\deg P \leq N-1}\deg P < \deg g \leq \sum_{\deg P \leq N}\deg P.$ Then $g$ has at most $\sum_{1 \leq k \leq N}\pi(N)$ irreducible factors. Therefore, since $\deg P_i \leq \deg \omega_i$, we have
$$\sum_{d|g}\frac{1}{q^{ \deg d }} \leq \prod_{\omega | g}\left(1 + \frac{1}{q^{\deg \omega}} \right) \leq \prod_{\deg P \leq N }\left(1+\frac{1}{q^{\deg P}}\right)=\prod_{1\leq k \leq N}\left(1+\frac{1}{q^k}\right)^{\pi(k)}.$$
Using $\pi(k) \leq \frac{q^k}{k}$ this is bounded by
$$\prod_{1\leq k \leq N}\left(1+\frac{1}{q^k}\right)^{\frac{q^k}{k}} \leq \prod_{1\leq k \leq N}e^{\frac{1}{k}} \leq e^{1+\log N} =  Ne.$$
Now we bound $N$ in terms of $\deg g$ as follows
$$\deg g > \sum_{\deg P \leq N-1}\deg p = \sum_{1 \leq k \leq N-1}\pi(k)k \geq \sum_{k | N -1} \pi(k) k = q^{N-1}$$
by the prime number theorem in $\mathbb{F}_q[t].$ This gives $N \leq 1+\frac{\log (\deg g)}{\log q}$ which completes the proof of the Lemma.
\end{proof}

\section{Proof of Theorem~\ref{thm}}

Let $ \theta \in \mathbb{T}$ and choose $g$ and $s$ as in Lemma~\ref{lem}. We start by giving an explicit description of a set a representatives for the equivalence relation $\mathcal{R}_{s,g}$. It is not hard to show that
$$\mathcal{S}_{s,g} = \{t^{[\frac{n}{2}]} g b_1 + b_2 \: | \: \deg b_1 = s, b_1 \text{ monic}, \deg b_2 < \deg g \}$$
is such a set. Furthermore, $$\mathcal{S}^*_{s,g} = \{t^{[\frac{n}{2}]} g b_1 + b_2 \: | \: \deg b_1 = s, b_1 \text{ monic}, \deg b_2 < \deg g, (b_2,g)=1 \}$$
defines a set of reduced representatives modulo $\mathcal{R}_{s,g}$. See~\cite{Hay2} Lemma 7.1 for details.

Then by Lemma~\ref{lem} and the orthogonality of characters modulo $\mathcal{R}_{s,g}^*$ we can write

\begin{align*}
\sum_{ \deg f = n} \mu(f) \textbf{e}_q(f \theta) &= \sum_{ b \in \mathcal{S}_{s,g} } \sum_{ \substack{ \deg f = n \\ f \equiv b \text{ mod } \mathcal{R}_{s,g} } } \mu(f) \textbf{e}_q(f \theta) \\
&= \sum_{d | g}\sum_{ \substack{ b \in \mathcal{S}_{s,g} \\ (g,b)=d} } \textbf{e}_q(b \theta) \sum_{ \substack{ \deg f = n \\ f \equiv b \text{ mod } \mathcal{R}_{s,g} } } \mu(f) \\
&= \sum_{d | g}\sum_{ \substack{ b \in \mathcal{S}_{s,g/d} \\ (g/d,b)=1} } \textbf{e}_q(bd \theta) \sum_{ \substack{ \deg f = n - \deg d \\ f \equiv b \text{ mod } \mathcal{R}_{s,g/d} } } \mu(fd) \\
&= \sum_{d | g}\sum_{  b \in \mathcal{S}^*_{s,g/d} } \textbf{e}_q(b d \theta)\sum_{ \deg f = n -\deg d} \frac{1}{q^s \phi(g/d)} \sum_{ \chi \text{ mod } \mathcal{R}^*_{s,g/d}} \overline{\chi}(b) \chi(f) \mu(fd).
\end{align*}

Notice that $\mu(fd) = \mu(f)\mu(d) \chi_d(f)$ where $\chi_d(f)$ is the trivial character modulo $\mathcal{R}_{s,d}^*$. We can therefore rewrite the above as
$$
=\sum_{d | g}\frac{\mu(d)}{q^s \phi(g/d)}\sum_{ \chi \text{ mod } \mathcal{R}^*_{s,g/d}}\left( \sum_{  b \in \mathcal{S}^*_{s,g/d} } \textbf{e}_q(b d \theta)\overline{\chi}(b)\right ) \left( \sum_{ \deg f = n -\deg d} \mu(f) \chi \chi_d(f) \right).
$$

Now $\chi$ is a character modulo $\mathcal{R}^*_{s,g/d}$ and $\chi_d$ is a character modulo $\mathcal{R}^*_{s,d}$. Therefore, $\chi\chi_d$ is a character modulo $\mathcal{R}^*_{s,g}$, and so using the triangle inequality and Lemma~\ref{mobchar} we can bound this in absolute value by
\begin{align*}
q^{n/2} \sum_{ \substack{ d | g \\ g \text{ square-free } }}\frac{1}{q^{s+\deg d/2} \phi(g/d)} {{n-\deg d+s+\deg g -2}\choose{s+ \deg g -2}}\sum_{ \chi \text{ mod } \mathcal{R}^*_{s,g/d}} \left| \sum_{ b \in \mathcal{S}_{s,g/d} } \textbf{e}_q(bd \theta) \overline{\chi}(b) \right|
\end{align*}
We bound the Gauss sum over $\chi$ mod $\mathcal{R}^*_{s,g/d}$ in the standard way using the Cauchy–Schwarz inequality and Parseval's identity as follows
\begin{align*}
\sum_{ \chi \text{ mod } \mathcal{R}^*_{s,g/d}} \left| \sum_{ b \in \mathcal{S}_{s,g/d} } \textbf{e}_q(bd \theta) \overline{\chi}(b) \right| & \leq \left(\sum_{\chi \text{ mod } \mathcal{R}^*_{s,g/d}}1  \sum_{ \chi \text{ mod } \mathcal{R}^*_{s,g/d} } \left|\sum_{b \in \mathcal{S}_{s,g/d}} \textbf{e}_q(bd \theta) \overline{\chi}(b) \right|^2  \right)^{1/2} \\
&= \left(  q^s \phi(g/d) \sum_{b_1, b_2 \in \mathcal{S}_{s,g/d} } \textbf{e}_q(d(b_1-b_2)\theta) \sum_{\chi \text { mod } \mathcal{R}^*_{s,g} }  \overline{\chi}(b_1) \chi(b_2) \right)^{1/2} \\
&= \left(  (q^s \phi(g/d))^2 \sum_{b_1 = b_2 \in \mathcal{S}^*_{s,g/d} } \textbf{e}_q((b_1-b_2)\theta) \right)^{1/2} \\
&=(q^s \phi(g/d))^{3/2}.
\end{align*}
Recall that $s+ \deg g = n-[\frac{n}{2}] \geq n/2$ so that
$${{n - \deg d+s+\deg g -2}\choose{s+ \deg g -2}} \leq {{2n - [\frac{n}{2}] -2}\choose{n - [\frac{n}{2}] -2}}.$$
We can bound this binomial coefficient using the fact that for all positive integers $k$,
$$\sqrt{2\pi} k^{k+\frac{1}{2}}e^{-k + \frac{1}{12k+1}}<k!<\sqrt{2\pi} k^{k+\frac{1}{2}}e^{-k + \frac{1}{12k}}.$$
This precise form of Stirling's formula is due to Robbins~\cite{Rob}. It follows that if $k=[\frac{n}{2}]$ then
$${{2n - [\frac{n}{2}] -2}\choose{n - [\frac{n}{2}] -2}}<{{3k}\choose{k}}<\frac{1}{\sqrt{2\pi}}e^{\frac{1}{36k}-\frac{1}{12k+1}-\frac{1}{24k+1}} \frac{(3k)^{3k+\frac{1}{2}}}{k^{k+\frac{1}{2}}(2k)^{2k+\frac{1}{2}}}<\frac{1}{\sqrt{4\pi k/3}}\left(\tfrac{3\sqrt{3}}{2}\right)^{2k}.$$
Putting it all together with $\phi(g/d) \leq q^{\deg g - \deg d}$ and Lemma~\ref{better} we get
\begin{align*}
\left| \sum_{ \deg f = n} \mu(f) \textbf{e}_q(f \theta) \right| &\leq q^{n/2}\frac{1}{\sqrt{2\pi (n-1)/3}}\left(\tfrac{3\sqrt{3}}{2}\right)^{n}\sum_{d | g} \frac{(q^s \phi(g/d))^{1/2}}{q^{\deg d /2}} \\
&\leq q^{n-\frac{1}{2}[\frac{n}{2}]} \frac{(1+\tfrac{\log n}{ \log q })e}{\sqrt{2\pi (n-1)/3}}\left(\tfrac{3\sqrt{3}}{2}\right)^{n}
\end{align*}
and Theorem~\ref{thm} easily follows after a short numerical calculation.


\begin{thebibliography}{HD}

\bibitem{BaHar}
R. C. Baker and G. Harman,
\textit{Exponential sums formed with the M\"obius function},
J. London Math. Soc.
(2)
43
(1991), no. 2, 193–198.

\bibitem{PiHo}
P.-Y. Bienvenu, T. H. L\^e,
\textit{Linear and Quadratic uniformity of the M\"obius function over $\mathbb{F}_q[t]$},
preprint,
arXiv:1711.05358.

\bibitem{Dav}
H. Davenport,
\textit{On some infinite series involving arithmetical functions (II)},
The Quarterly Journal
of Mathematics
, 8(1):313–320, 1937.

\bibitem{Hay1}
D. R. Hayes,
\textit{The distribution of irreducibles in GF[q, x]}, Trans. Amer. Math. Soc.
117
(1965), 101–127.

\bibitem{Hay2}
D. R. Hayes,
\textit{The expression of a polynomial as a sum of three irreducibles},
Acta Arith. 11 (1966) 461–488.

\bibitem{TH}
T. H. L\^e,
\textit{Green-Tao theorem in function fields},
Acta Arith. 147 (2011).

\bibitem{Pol}
P. Pollack,
\textit{Irreducible polynomials with several prescribed coefficients},
Finite Fields Appl. 22 (2013) 70–78.

\bibitem{Rhin}
G. Rhin,
\textit{R\'epartition modulo 1 dans un corps de s\'eries formelles sur un corps fini},
Dissertationes Math. (Rozprawy Mat.)
95
(1972), 75.

\bibitem{Rob}
H. Robbins,
\textit{A Remark on Stirling's Formula},
The American Mathematical Monthly. 62 (1955) 26–29.

\end{thebibliography}
\end{document}